\documentclass[10pt]{amsart}
\title{Quotients of degenerate Sklyanin algebras}
\author{Kevin De Laet}
\address{Department of Mathematics, University of Antwerp \\ 
 Middelheimlaan 1, B-2020 Antwerp (Belgium) \\ {\tt kevin.delaet2@uantwerpen.be}}
\date{}
\usepackage{amsmath,amsfonts,array,amscd,amsthm,amssymb,stackrel,verbatim,wrapfig,subfig,float}
\usepackage{enumerate,filecontents}
\usepackage{url}
\usepackage{tikz}
\usepackage[all]{xy}
\usetikzlibrary{arrows,matrix}
\usetikzlibrary{shapes.geometric}
\usetikzlibrary{decorations.markings} 

\tikzset{
  vertice/.style={circle,draw=black},
  decoration={markings,mark=at position 0.5 with {\arrow{>}}}
}

\theoremstyle{plain}
\newcommand{\wis}[1]{{\text{\em \usefont{OT1}{cmtt}{m}{n} #1}}}
\newcommand{\C}{\mathbb{C}}
\newcommand{\N}{\mathbb{N}}

\newcommand{\Z}{\mathbb{Z}}
\newcommand{\A}{\mathbb{A}}
\newcommand{\PP}{\mathbb{P}}

\newtheorem{theorem}{Theorem}[section]
\newtheorem{lemma}[theorem]{Lemma}
\newtheorem{proposition}[theorem]{Proposition}
\newtheorem{corollary}[theorem]{Corollary}

\theoremstyle{definition}
\newtheorem{mydef}[theorem]{Definition}

\DeclareMathOperator{\im}{Im}

\DeclareMathOperator{\Grass}{Grass}
\DeclareMathOperator{\Emb}{Emb}

\numberwithin{equation}{section}
\begin{document}
\sloppy
\begin{abstract}
In this paper it is shown how the Heisenberg group of order 27 can be used to construct quotients of degenerate Sklyanin algebras. These quotients have properties similar to the classical Sklyanin case in the sense that they have the same Hilbert series, the same character series and a central element of degree 3. Regarding the central element of a 3-dimensional Sklyanin algebra, a better way to view this using Heisenberg-invariants is shown.
\end{abstract}
\maketitle
\tableofcontents
\section{Introduction}
The 3-dimensional Sklyanin algebras form an important class of noncommutative graded algebras, as they correspond to the notion of a noncommutative $\PP^2$ following Artin, Tate, Van den Bergh and others (see for example \cite{ATV1} and \cite{ATV2}). These algebras are parametrized by an elliptic curve $\xymatrix{E \ar@{^{(}->}[r] & \PP^2}$ and a point $\tau \in E$. They form the largest class of examples of quadratic 3-dimensional AS-regular algebras, that is, graded algebras of global dimension 3 with relations in degree 2 with excellent homological properties. These AS-regular algebras can be described as the quotient of $ \C\langle x,y,z\rangle$ by the relations
\begin{equation}
\begin{cases}
a yz + b zy + c x^2,\\
a zx + b xz + c y^2,\\
a xy + b yx + c z^2,
\end{cases}
\label{eq:Sklyanin}
\end{equation}
with $[a:b:c]\in \PP^2$ not one of the 12 points of
\begin{eqnarray}
\{[0:0:1],[0:1:0],[1:0:0]\}\cup \{[a:b:c]\in \PP^2 | a^3=b^3=c^3=1\}.
\label{eq:nonreg}
\end{eqnarray}
It was remarked in many early papers (see for example \cite{ASregular}) about these algebras that there was a central element of degree 3, which was somewhat mysterious. In \cite{DeLaetLeBruyn}, a intrinsic presentation of this central element was found. It turns out that the central element gave a connection between the Sklyanin algebra $\mathcal{A}_\tau(E)$ and the algebra $\mathcal{A}_{-2\tau}(E)$ if $E$ is an elliptic curve and $\tau \in E$. This was proved using the concept of superpotentials, as explained in for example \cite{Walton}.
\par The first purpose of this paper is to give a better statement of this theorem in terms of Heisenberg invariants and the degree 3 part of the Koszul dual of a Sklyanin algebra.
\par The second purpose is the study of quotients of the 12 nonregular algebras. In particular, we show that there is a 1-dimensional family of quotients of each of these 12 algebras parametrized by $\C^*$ such that the quotients have Hilbert series $\frac{1}{(1-t)^3}$. In addition, these algebras also have a central element of degree 3, fixed by the Heisenberg group. We also show that for the constructed quotients of the algebra $\C\langle x,y,z\rangle/(x^2,y^2,z^2)$, the $n$th roots of unity give quotients that are finite modules over their center of PI-degree $2n$.
\subsection{Notation} In this article, we use the following notations:
\begin{itemize}
\item $\mathbf{V}(I)$ for $I \subset \C[a_1,\ldots,a_n]$ an ideal is the Zariski-closed subset of $\A^n$ or $\PP^{n-1}$ determined by $I$, it will be clear from the context if the projective or affine variety is used. 
\item $D(I)$ for $I$ an ideal $I \subset \C[a_1,\ldots,a_n]$ is the open subset $\A^n \setminus \mathbf{V}(I)$ or $\PP^{n-1} \setminus \mathbf{V}(I)$, it will be clear from the context if it is an open subset of affine space or of projective space. If $I = (a)$, then we write $D(a)$ for $D(I)$.
\item $\Z_n = \Z/n\Z$ for $n \in \mathbb{N}$.
\item $\Grass(m,n)$ will be the projective variety parametrizing $m$-dimensional vector spaces in $\C^n$.
\item For an algebra $A$ and elements $x,y \in A$, $\{x,y\} = xy + yx$ and $[x,y] = xy-yx$.
\item For $V$ a $n$-dimensional vector space, we set $T(V) = \oplus_{k=0}^\infty V^{\otimes k}$, the tensor algebra over $V$.
\item Every graded algebra $A$ will be positively graded, finitely generated over $\C$ and connected, that is $A_0 = \C$.
\item The group $\wis{SL}_m(p)$ (respectively $\wis{PSL}_m(p)$) is the special linear group (respectively projective special linear group) of degree $m$ over the finite field with $p$ elements.
\item For any vector space $V$, $\C[V] = T(V)/(wv-vw|w,v \in V)$.
\item If $A$ is a connected, finitely generated, positively graded algebra, then the Hilbert series is defined as $H_A(t) = \sum_{k=0}^\infty \dim A_k t^k$.
\item Take a reductive group $G$ and 2 finite dimensional representations $V,W$ of $G$. Then $\Emb_G(V,W)$ is the set of injective linear $G$-maps from $V$ to $W$ up to $G$-isomorphisms of $V$.
\item Modules will always be left modules unless otherwise mentioned.
\end{itemize}

\section{$G$-algebras}
\label{sec:Galg}
This section is a summary of the general theory developed in \cite{DeLaet2}.
\begin{mydef}
Let $G$ be a reductive group. We call a positively graded connected algebra $A$, finitely generated in degree 1, a \textit{$G$-algebra} if $G$ acts on it by gradation preserving automorphisms.
\end{mydef}
This definition implies that there exists a finite dimensional representation $V$ of $G$ such that $T(V)/I \cong A$ with $I$ a graded ideal of $T(V)$, which is itself a $G$-subrepresentation of $T(V)$.
\par One can make quadratic $G$-algebras as follows. Let $V$ be a $G$-representation. Then $V \otimes V$ is also a $G$-representation which decomposes as a summation of simple representations, say $V \otimes V \cong \oplus_{i=1}^m S_i^{a_i}$ where the $S_i$ are distinct simple representations of $G$ and $a_i \geq 0$. A $G$-algebra $A$ is then constructed by taking embeddings of the $S_i$ in $V \otimes V$ as relations of $A$.
\par One can of course do the same for other degrees by taking relations in $T(V)_i = V^{\otimes i} \cong  \oplus_{j=1}^m S_j^{a_j}$ and take different embeddings of the simple representations of $G$ in $V^{\otimes i}$ as relations.
\begin{mydef}
Let $A$ be a $G$-algebra with corresponding ideal $I$ of $T(V)$. We call $B$ a \textit{$G$-deformation} of $A$ up to degree $k$ if $B$ is also a quotient of $T(V)$ such that $\forall 1 \leq i \leq k: A_i \cong B_i$ as $G$-representations. We will call $B$ a $G$-deformation if $\forall i \in \N: A_i \cong B_i$ as $G$-representations.
\end{mydef}
If the relations for $A$ are all of the same degree $k$, then all $G$-deformations up to degree $k$ of $A$ depend on a product of Grassmannians. For example, let $A$ be a quadratic algebra of which we want to find all $G$-deformations up to degree 2. Let $I_2 = \oplus_{i=1}^m S_i^{e_i} \subset V \otimes V = \oplus_{i=1}^m S_i^{a_i}$ with $S_i$ distinct simple representations and $0 \leq e_i \leq a_i$ natural numbers. Then the $G$-deformations up to degree 2 are parametrized by $\Emb_G(\oplus_{i=1}^m S_i^{e_i},\oplus_{i=1}^m S_i^{a_i})=\prod_{i=1}^m \Grass(e_i,a_i)$.
\par In general, the total set of $G$-deformations up to degree $k$ of a $G$-algebra $A = T(V)/I$ are determined by a Zariski closed subset of 
$$
Z_k=\prod_{j=1}^k \prod_{S_i \text{ simple}}\Grass(e_{i,j},a_{i,j})
$$
with $I_j = \oplus_{S_i \text{ simple}} S_i^{e_{i,j}} \subset T(V)_i = \oplus_{S_i \text{ simple}} S_i^{a_{i,j}}$
\begin{mydef}
We say that a variety $Z$ \textit{parametrizes $G$-deformations up to degree $k$} of a $G$-algebra $A$ if $\xymatrix{Z \ar@{^{(}->}[r]^-\phi & Z_k}$ can be embedded in $Z_k$ and the point corresponding to $A$ in $Z_k$ belongs to the image of $\phi$. We say that $Z$ \textit{parametrizes $G$-deformations} of $A$ if $Z$ parametrizes $G$-deformations up to degree $k$ for some $k$ and for each point $x \in Z$ with corresponding algebra $A_x$, we have
$$
\forall i \in \N: (A_x)_i \cong A_i \text{ as $G$-representations}.
$$
\end{mydef}
We will show in the next section that the 3-dimensional Sklyanin algebras are $H_3$-deformations of the polynomial ring $\C[V]$. We first show a computational way to decode how a $G$-algebra decomposes as a $G$-module.

\subsection{Character series}
\label{sec:kos}
Given a $G$-algebra $A$, it is a natural question to ask how $A$ behaves as a $G$-module. As $G$ acts as gradation preserving automorphisms, we have a decomposition
$$
A = \bigoplus_{k=0}^{\infty} \bigoplus_{S \text{ simple}} S^{e_{k,S}}
$$
with almost all $e_{k,S}$ equal to 0. We will only consider the case that $G$ is finite.

\begin{mydef}
Let $G$ be a finite group. The \textit{character series} for an element $g \in G$ and for a $G$-algebra $A$ is a formal sum 
\begin{displaymath}
Ch_A(g,t) = \sum_{n \in \mathbb{Z}} \chi_{A_n}(g) t^n.
\end{displaymath}
\end{mydef}
For example, if $g =  1$, then we have $Ch_A(1,t) = H_A(t)$, the Hilbert series of $A$. As a character of a representation is constant on conjugacy classes, we can represent the decomposition of $A$ in simple $G$-representations as a vector of length equal to the number of conjugacy classes and on the $i$th place the character series $Ch_A(g,t)$ with $g \in C_i$, the $i$th conjugacy class.
\begin{lemma}
Let $V$ be a simple representation of $G$ and let $A$ be a $G$-algebra constructed from $T(V)$. For every element $z$ of the center, we have that $Ch_A(z,t) = H_A(\lambda t)$, where $z$ acts on $V$ by multiplication with $\lambda$.
\label{lem:cent}
\end{lemma}
\begin{proof}
It follows that in degree $k$ the action of $z$ on $A_k$ is given by multiplication with $\lambda^k$, so the character series for the element $z$ in this case is given by $$Ch_A(z,t)=\sum_{k=0}^\infty \lambda^k \dim A_k t^k = H_A(\lambda t).$$ 
\end{proof}

\section{The finite Heisenberg group of order $27$}
While in previous papers (see \cite{DeLaet} and \cite{DeLaet2}) we needed the finite Heisenberg group of order $p^3$ for any odd prime $p$, we will only consider here the special case $p=3$.
\begin{mydef}
The \textit{Heisenberg group of order 27} is the finite group given by generators and relations
\[
H_3=\langle e_1, e_2 |~[e_1,e_2]\text{ central},~e_1^3=e_2^3=1 \rangle. \]
\end{mydef}
$H_3$ is a central extension of the group $\Z_3 \times \Z_3$ with $\Z_3$,
\begin{eqnarray}
 \xymatrix{1 \ar[r]& \mathbb{Z}_3 \ar[r] & H_3 \ar[r]&\mathbb{Z}_3 \times \mathbb{Z}_3 \ar[r]&  1}.
 \label{eqn:centralext}
\end{eqnarray}
$H_3$ has 9 1-dimensional representations coming from the quotient $H_3/([e_1,e_2]) = \Z_3 \times \Z_3$ and 2 3-dimensional simple representations, corresponding to the  primitive 3rd roots of unity. These 2 representations are defined in the following way: let $\omega $ be a primitive 3rd root of unity and let $V_1=\C^3 = \C x_0 + \C x_1 + \C x_2$. Then the action is defined by
\begin{eqnarray}
e_1 \cdot x_i = x_{i-1},& e_2 \cdot x_i = \omega^i x_i.
\end{eqnarray}
$V^*$ is the representation corresponding to $\omega^2$ and will be denoted by $V_2$. We will use $\chi_{a,b}$ for the character $\xymatrix{H_3 \ar[r]^-{\chi_{a,b}}& \C}$ defined by \begin{eqnarray}
\chi_{a,b}(e_1) = \omega^a, & \chi_{a,b}(e_2) = \omega^b.
\end{eqnarray}
There is an action of $\wis{SL}_2(3)$ on $H_3$ as group automorphisms by the rule
$$
\begin{bmatrix}
A & B \\ C & D
\end{bmatrix}\cdot e_1 = e_1^A e_2^C, 
\begin{bmatrix}
A & B \\ C & D
\end{bmatrix}\cdot e_2 = e_1^B e_2^D.
$$
The central element $[e_1,e_2]$ is fixed by this action. From this it follows that the induced action of $\wis{SL}_2(3)$ on the simple representations fixes $V_1$ and $V_2$.
\subsection{3-dimensional Sklyanin algebras are $H_3$-deformations}
Using the construction of Section \ref{sec:Galg}, we will now show that the 3-dimensional Sklyanin algebras are $H_3$-deformations of $\C[V]$ with $V= V_1$ as defined above.
\par Write $ \mathcal{A}=\C[V] = T(V)/(V\wedge V)$. In order to find all $H_3$-deformations up to degree 2 of $\mathcal{A}$, we need to decompose $V \otimes V$ in $H_3$-representations. A quick calculation shows that $V \otimes V = (V^*)^3$ and that the $H_3$-generators of $V \otimes V$ are given by
\begin{eqnarray}
x_1 x_2 - x_2 x_1, & x_1 x_2 + x_2 x_1, & x_0^2.
\end{eqnarray}
The first generator corresponds to the wedge product $V \wedge V$. Taking another copy of $V^*$ in $V\otimes V$ corresponds to an element of $\wis{Grass}(1,3)$. Given $p=[A:B:C] \in \wis{Grass}(1,3) = \PP^2$, then $p$ determines the quotient $\C\langle x_0,x_1,x_2 \rangle/(I)$ with $I$ generated by the relations
$$
\begin{cases}
A(x_1 x_2 - x_2 x_1) + B(x_1 x_2 + x_2 x_1) + C(x_0^2),\\
A(x_2 x_0 - x_0 x_2) + B(x_2 x_0 + x_0 x_2) + C(x_1^2),\\
A(x_0 x_1 - x_1 x_0) + B(x_0 x_1 + x_1 x_0) + C(x_2^2).
\end{cases}
$$
Putting $a = A+B$, $b = B-A$ and $c=C$, one gets the familiar relations of the 3-dimensional Sklyanin algebras as in equation \ref{eq:Sklyanin}.
\par In particular, the 3-dimensional Sklyanin algebras have the same Hilbert series as the polynomial ring in 3 variables. Using the results of \cite{DeLaet2}, we see that the character series with respect to $H_3$ of any Sklyanin algebra $\mathcal{B}$ is the same as the character series of $\mathcal{A}$. In fact, this is true for all Artin-Schelter regular algebras parametrized by points of $\Emb_{H_3}(V^*,(V^*)^3)$.
\par In addition, the $\wis{SL}_2(3)$ (left) action on $H_3$ induces a (right) action on $\Emb_{H_3}(V^*,(V^*)^3) \cong \PP^2$. This projective representation of $\wis{SL}_2(3)$ has the property that points lying in the same orbit determine isomorphic algebras. In particular, the only non-regular algebras are those lying in the $\wis{SL}_2(3)$-orbit of either the point $[1:0:0]$ or $[0:0:1]$. For more information, see amongst others \cite{abdelgadir2014compact}, \cite{DeLaet}.
\section{Central elements and Heisenberg invariants}
\label{sec:central}
In \cite{DeLaetLeBruyn}, it was proved that there was a connection between the superpotential defining the Sklyanin algebra $\mathcal{A}_{-2\tau}(E)$ and the central element $c_3$ of degree 3 in $\mathcal{A}_{\tau}(E)$. This connection however can better be explained using Heisenberg invariant elements of $V \otimes V \otimes V$.
\par Let $p=[a:b:c]$ and define $S_p = \C \langle x,y,z \rangle /(R_p)$ to be the algebra with relations
$$R_p=\begin{cases}
a yz + b zy + c x^2,\\
a zx + b xz + c y^2,\\
a xy + b yx + c z^2.
\end{cases}
$$
Let $W_p \subset V \otimes V$ be the vector space spanned by these relations, with $V = \C x + \C y +\C z$. Then $W_p \otimes V \cap V \otimes W_p$ is generically a 1-dimensional vector space, generated by
$$
a(zxy+xyz+yzx) + b(yxz+zyx+xzy) + c (x^3+y^3+z^3).
$$
This element is easily seen to be fixed by $H_3$. In turn, any element $g \in \PP((V \otimes V \otimes V)^{H_3})$ determines quadratic relations by taking the cyclic derivatives $\delta_x,\delta_y,\delta_z$ as in \cite{DeLaetLeBruyn}. For AS-regular algebras, $W_p \otimes V \cap V \otimes W_p$ is 1-dimensional, so on the open subset $\PP^2 \setminus D$ with $D = \bigcup_{g \in \wis{SL}_2(3)} g \cdot\{[1:0:0],[0:0:1]\}$, we have an injective morphism
$$
\xymatrix{\PP^2\setminus D \ar[r]^-\phi& \PP((V \otimes V \otimes V)^{H_3}) = \PP^2}
$$
In order for this to extend to $\Emb_{H_3}(V^*,(V^*)^3) = \PP^2$ and to get an isomorphism, one should need that $(W_p \otimes V \cap V \otimes W_p)^{H_3}$ is always 1-dimensional. This is indeed the case. Let $\chi_{a,b}$ be the 1-dimensional representation of $H_3$ defined by $\chi_{a,b}(e_1) = \omega^a, \chi_{a,b}(e_2) = \omega^b$.
\begin{theorem}
Let $\mathbb{V}=\mathbf{V}(abc)\subset \PP^2_{[a:b:c]}$. Then for each vertex $p$ of $\mathbb{V}$, the decomposition of $W_p\otimes V \cap V \otimes W_p$ in $H_3$-representations is given by $\chi_{0,0}\oplus \chi_{1,0}\oplus\chi_{2,0}$. In particular,
$(W_p\otimes V \cap V \otimes W_p)^{H_3}$ is 1-dimensional.
\end{theorem}
\begin{proof}
As these algebras are monomial algebras, it is easy to find a basis of $W_p\otimes V \cap V \otimes W_p$. We have
\begin{itemize}
\item for $[0:0:1]$, we find $\C x^3+\C y^3+\C z^3$,
\item for $[1:0:0]$, we find $\C zxy+\C xyz+\C yzx$,
\item for $[0:1:0]$, we find $\C yxz+\C zyx+\C xzy$.
\end{itemize} 
In these 3 cases, it is clear that $e_2$ works trivially on these elements and $e_1$ works by cyclic permutation, leading to the claimed decomposition.
\end{proof}
Now, the other points that correspond to nonregular algebras lie in the $\wis{SL}_2(3)$-orbit of these 3 points. In order to prove that $(R_p \otimes V \cap V \otimes R_p)^{H_3}$ is indeed 1-dimensional, we need to work out what the $\wis{SL}_2(3)$-orbits are in the  set of simple representations of $H_3$. From the natural action of $\wis{SL}_2(3)$ on $\Z_3 \times \Z_3$ we find that
\begin{itemize}
\item $\chi_{00}$ is fixed,
\item $V$ and $V^*$ are fixed because the center is fixed,
\item the action is transitive on the set $\chi_{a,b}$, $(a,b) \neq (0,0)$.
\end{itemize}
Now, the vector spaces of the theorem have one thing in common: the action of $e_2$ is fixed. If one takes the $H_3$-representations $W_{a,b} = \chi_{a,b} \oplus \chi_{-a,-b}, a,b \in \{0,1\} $, then the center of $\wis{SL}_2(3)$ works trivially on the set $\{W_{a,b}|a,b \in \{0,1\}  \}$, so the action is really a $\wis{PSL}_2(3)$-action.
\par From this, we see that the induced action of $\wis{SL}_2(3)$ on the decomposition of $W_{[1:0:0]} \otimes V \cap V \otimes W_{[1:0:0]}$ or $W_{[0:0:1]} \otimes V \cap V \otimes W_{[0:0:1]}$ sends $\chi_{0,0}\oplus W_{1,0}$ to $\chi_{0,0} \oplus W_{a,b}$ for some  $a,b \in \{0,1\}$.
We have proved
\begin{theorem}
$(W_p\otimes V \cap V \otimes W_p)^{H_3}$ is 1-dimensional for every $p \in \PP^2$.
\end{theorem}
So $\phi$ indeed extends to an isomorphism of $\PP^2$.
\par Now, theorem 1 of \cite{DeLaetLeBruyn} can be described as
\begin{theorem}
Let $A = T(V)/R_p$, $p=[a:b:c]$ be a Sklyanin algebra, $c_3$ be the central element of degree 3 in $A$ and let
$$
\xymatrix{(V \otimes V \otimes V)^{H_3} \ar[r]^-\pi & A_3^{H_3} }
$$
be the natural projection map. Then $\pi^{-1}(\C c_3)$ is a 2-dimensional vector space of $(V \otimes V \otimes V)^{H_3}$, which corresponds in $\PP((V \otimes V \otimes V)^{H_3})$ to the tangent line of the elliptic curve $E$ at the point $p$.
\end{theorem}
\begin{proof}
According to \cite{DeLaetLeBruyn}, the vector space corresponding to the point $-2p \in \PP((V \otimes V \otimes V)^{H_3})$ is indeed mapped to $\C c_3$. As the vector space $$
\C(a(zxy+xyz+yzx) + b(yxz+zyx+xzy) + c (x^3+y^3+z^3))
$$
is the kernel of $\pi$, it follows that the vector space generated by $p$ and $-2p$ is indeed $\pi^{-1}(\C c_3)$. The fact that this is the tangent line to $p$ at $E$ follows as the third point of intersection of the line through $p$ and $-2p$ is $p$ itself.
\end{proof}
\section{Quotients of non-regular quadratic algebras}
In the projective plane $\PP^2_{[a:b:c]}$, there are 12 points where the corresponding algebra is not AS-regular: the $\wis{SL}_2(3)$-orbit of $[1:0:0]$ (containing 8 elements) and the orbit of $[0:0:1]$ (containing 4 elements). All these algebras have as Hilbert series $\frac{1+t}{1-2t}$ and are clearly not domains, for a detailed description of these algebras, see \cite{smithdegenerate} and \cite{DegenerateWalton}. However, it seems that these algebras have a 1-dimensional family of quotients that `behave' like the 3-dimensional AS-regular algebras in the following sense:
\begin{itemize}
\item the Hilbert series is the same,
\item the character series is the same for each element of $H_3$ and
\item there exists a central element of degree 3, fixed by the $H_3$-action.
\end{itemize}
Let us consider the following example: take the Clifford algebra $C$ over $\C[u_0,u_1,u_2]$ with associated quadratic form
$$
\begin{bmatrix}
0 & u_2 & u_1 \\ u_2 & 0 & u_0 \\ u_1 & u_0 & 0
\end{bmatrix}.
$$
In terms of generators and relations of $C$, we have 3 generators $x_0,x_1,x_2$ with relations
$$
\begin{cases}
x_0^2 = x_1^2 = x_2^2 = 0,\\
[\{x_i,x_{i+1}\},x_{i+2}] = 0, 0 \leq i \leq 2.
\end{cases}
$$
This algebra is a quotient of the algebra $S_{[0:0:1]}$ by 2 elements of degree 3 (adding 2 commutation relations of degree 3 automatically implies the third relation).
\begin{theorem}
The character series of $C$ is the same as the character series of the polynomial ring in 3 variables.
\end{theorem}
\begin{proof}
Define on $C$ an action of $H_3$ by \begin{eqnarray*}
 e_1 \cdot x_i = x_{i-1}         & e_1 \cdot u_i = u_{i-1},\\
 e_2 \cdot x_i = \omega^i x_{i}  & e_2 \cdot u_i = \omega^{2i}u_{i}.
\end{eqnarray*}
$C$ is a free module of rank $8$ over $\C[u_0,u_1,u_2]$ with basis
$$
\{1,x_0,x_1,x_2,x_1x_2,x_2x_0,x_0x_1,x_0x_1x_2\}.
$$
It is then easy to compute that under these conditions, $C$ is a graded algebra with character series equal to the polynomial ring in 3 variables with the standard action of $H_3$ (see \cite{DeLaet} for the AS-regular case, this one is similar).
\end{proof}
In particular, this means that the natural epimorphism
$$
\xymatrix{S_{[0:0:1]} \ar[r]^-\pi & C}
$$
has as kernel an ideal generated by 2 elements of degree 3. The $\C$-vector space generated by these 2 elements decomposes as $H_3$-representations into $\chi_{1,0} \oplus \chi_{-1,0}$, which is the `$H_3$-surplus' of $S_{[0:0:1]}$ in degree 3.
\par Considering the representations of $C$, we have
\begin{theorem}
$C$ is Azumaya over every point of $\wis{Spec}(Z(C))$ except over the trivial ideal $(u_0,u_1,u_2)$.
\end{theorem}
\begin{proof}
$C$ is not Azumaya over a point if and only if the associated quadratic form after specialization is of rank $\leq 1$. Taking the $2 \times 2$-minors of the quadratic form, this only happens if $u_i^2$ is mapped to $0$ for all $0 \leq i \leq 2$.
\end{proof}
This means that, considering representations, $C$ has more in common with the Sklyanin algebras associated to points of order 2 than the quantum algebra $\C_{-1}[x,y,z]$, although the last one is AS-regular.
\par In the next section, we will find a 1-dimensional family of quotients of $S_{[0:0:1]}$ by degree 3 elements that have the correct character series up to degree 4. we will show that there is an open subset in this quotient that gives algebras with the correct character series.
\subsection{The Hilbert series}
We work for now in the algebra $S=S_{[0:0:1]}$. $V$ will be the vector space $\C x + \C y + \C z$ and the action of $H_3$ is defined by
\begin{eqnarray}
e_1 \cdot x = z, & e_1 \cdot y = x, & e_1 \cdot z = y,\\
e_2 \cdot x = x, & e_2 \cdot y = \omega y, & e_2 \cdot z = \omega^2 z.
\end{eqnarray}

Decomposing the degree 3 part $S_3$ in $H_3$-modules, we find
$$
S_3 \cong \sum_{i,j \in \Z_3} \chi_{i,j} \oplus \chi_{0,0} \oplus \chi_{1,0} \oplus \chi_{2,0}.
$$
In particular, the multiplicity of $\chi_{1,0}$ and $\chi_{2,0}$ is 2 in $S_3$. This means that the variety parametrizing quotients of $S$ with the right character series up to degree 3 is given by $\PP^1 \times \PP^1$.
\par Let $I_p$ be the ideal of $S$ generated by
\begin{align*}
v_1 = A_1 (zxy+\omega xyz+ \omega^2 yzx) + B_1 (yxz+\omega zyx+\omega^2 xzy),\\
v_2 = A_2 (zxy+\omega^2 xyz+ \omega yzx) + B_2 (yxz+\omega^2 zyx+\omega xzy),
\end{align*}
with $p=([A_1:B_1],[A_2:B_2]) \in \PP^1 \times \PP^1$. Let $W_p = \C v_1 + \C v_2$. One of the similarities we want to investigate is whether we can get quotients such that the Hilbert series is correct, in particular, correct in degree 4. As $\dim S_4 = 24$ and $\dim \C[x,y,z]_4 = 15$, we need to have that $\dim (I_p)_4 = 9$. We have
$$
\dim (I_p)_4 = \dim W_p \otimes V + \dim V \otimes W_p - \dim V \otimes W_p \cap W_p \otimes V.
$$
We don't need to worry about character series up to degree 4 as $S_4 \cong V^8$. We need to find $W_p$ such that $\dim V \otimes W_p \cap W_p \otimes V$ is 3-dimensional. However, as this vector space is an $H_3$-representation and isomorphic to $V^e$ for some $e \in \N$, it is enough to find $W_p$ such that $(V \otimes W_p \cap W_p \otimes V)^{e_2}$ is 1-dimensional.
\begin{lemma}
$S/I_p$ has the correct Hilbert series up to degree 4 iff $p$ lies on the line $\mathbf{V}(A_1B_2-A_2B_1) = \Delta \subset \PP^1 \times \PP^1$.
\label{lem:Hil}
\end{lemma}
\begin{proof}
The elements fixed by $e_2$ in $W_p \otimes V+ V \otimes W_p$ lie in the vector space generated by
\begin{align*}
xv_1 &= A_1 xzxy + A_1 \omega^2 xyzx + B_1 xyxz + B_1 \omega   xzyx,\\
xv_2 &= A_2 xzxy + A_2 \omega   xyzx + B_2 xyxz + B_2 \omega^2 xzyx,\\
v_1x &= A_1 zxyx + A_1 \omega   xyzx + B_1 yxzx + B_1 \omega^2 xzyx,\\
v_2x &= A_2 zxyx + A_2 \omega^2 xyzx + B_2 yxzx + B_2 \omega   xzyx.
\end{align*}
Then $V \otimes W_p \cap W_p \otimes V \neq 0$ iff the following matrix has rank $\leq 3$:
$$
\begin{bmatrix}
A_1 & 0    & A_1 \omega^2 & B_1 & 0   & B_1 \omega \\
A_2 & 0    & A_2 \omega   & B_2 & 0   & B_2 \omega^2 \\
0   & A_1  & A_1 \omega   & 0   & B_1 & B_1 \omega^2 \\
0   & A_2  & A_2 \omega^2 & 0   & B_2 & B_2 \omega
\end{bmatrix}.
$$
The first $4\times 4$-minor is equal to
$$A_1B_2A_1A_2(\omega^2-\omega)-A_2B_1 A_1A_2(\omega^2-\omega) = (\omega^2-\omega)A_1A_2(A_1B_2-A_2B_1).$$
If $A_1 = 0$, then we can take $B_1=1$ and the $4 \times 4$-minor given by the columns $(2,3,4,5)$ becomes
$$
\begin{bmatrix}
0 & 0 & 1 & 0 \\
0 & A_2 \omega & B_2 & 0\\
0 & 0 & 0 & 1 \\
A_2 & A_2 \omega^2 & 0 & B_2
\end{bmatrix}.
$$
Taking the determinant of this matrix, one gets $A_2^2 \omega$. So this also implies $A_1B_2-A_2B_1=0$. A similar result is true if we set $A_2=0$.
\par If $A_1B_2-A_2B_1=0$, then one immediately checks that 
$$ \begin{bmatrix}
-B_2 & B_1 & -B_2 & B_1 \\
-A_2 & A_1 & -A_2 & A_1
\end{bmatrix}
\begin{bmatrix}
A_1 & 0    & A_1 \omega^2 & B_1 & 0   & B_1 \omega \\
A_2 & 0    & A_2 \omega   & B_2 & 0   & B_2 \omega^2 \\
0   & A_1  & A_1 \omega   & 0   & B_1 & B_1 \omega^2 \\
0   & A_2  & A_2 \omega^2 & 0   & B_2 & B_2 \omega
\end{bmatrix}=0.
$$
As either one of the rows of the first matrix is not 0, we have that $\dim V \otimes W_p \cap W_p \otimes V \geq 3$. The only points where this inequality is possibly strict is when either both $A_1$ and $B_2$ are 0 or both $A_2$ and $B_1$ are 0, this follows from taking the determinants of
$$
\begin{bmatrix}
A_2 & 0    & A_2 \omega \\
0   & A_1  & A_1 \omega\\
0   & A_2  & A_2 \omega^2
\end{bmatrix} \text{ and }
\begin{bmatrix}
B_2 & 0   & B_2 \omega^2 \\
0   & B_1 & B_1 \omega^2 \\
0   & B_2 & B_2 \omega
\end{bmatrix}.
$$
But if for example $A_1 = 0$, then necessarily $A_2 = 0$, but $B_2$ and $A_2$ can not be 0 at the same time. Similar results hold for the other cases, so we are done.
\end{proof}
This means that the only points we have to consider lie on the diagonal $\Delta \subset \PP_{[A_1:B_1]}^1 \times \PP_{[A_2:B_2]}^1$. From now on, we write $t = \frac{B}{A}$ for $[A:B] \in \PP^1$ and let $I_t$ be the ideal in $S$ generated by the elements
\begin{align*}
(v_1)_t = A (zxy+\omega xyz+ \omega^2 yzx) + B (yxz+\omega zyx+\omega^2 xzy),\\
(v_2)_t = A (zxy+\omega^2 xyz+ \omega yzx) + B (yxz+\omega^2 zyx+\omega xzy),
\end{align*}
\par The next obvious question is whether these algebras parametrized by $\PP^1$ have the correct Hilbert series. For $\C^* = \PP^1\setminus \{0,\infty \}$ this is true.
\begin{lemma}
The Clifford algebra $C$ corresponds to taking the quotient for the value $t=-1$.
\end{lemma}
\begin{proof}
It is enough to prove that the relation $[\{x,y\},z]$ belongs to the vector space $\C(v_1)_{-1}+\C(v_2)_{-1}$. This means that the following matrix should have rank 2
$$
\begin{bmatrix}
1  & \omega   & \omega^2 & -1 & -\omega   & -\omega^2 \\
1  & \omega^2 & \omega   & -1 & -\omega^2 & -\omega   \\
-1 & 1        & 0        &  1 & -1        & 0  
\end{bmatrix},
$$
which is indeed true.
\end{proof}
We will later see that all these algebras can be embedded in a smashed product $C \# \Z$ if $t \neq 0,\infty$, from which it will follow that
\begin{theorem}
Each algebra $S/I_t$ has the correct Hilbert series if $t \neq 0,\infty$.
\label{th:Hilbert}
\end{theorem}
\subsection{Point modules of $T_t$}
In \cite{smithdegenerate} the point modules of $S$ were classified. We can classify the point modules of $T_t=S/I_t$ using these results. Recall that point modules of $S$ were determined by the following way: let $\mathbb{V}=\mathbf{V}(XYZ) \subset \PP^2$ be the union of 3 lines and let $q_0,q_1,q_2$ be the intersection points. Then a point module $P$ of $S$ depends on a point sequence $p_0 p_1 p_2 \ldots$ fulfilling the following requirements
\begin{itemize}
\item for every $i \in \N$, $p_i \in \mathbb{V}$,
\item if $p_i \neq q_j$ for any $j$, then $p_{i+1}$ is the intersection point not lying on the same line as $p_i$,
\item if $p_i = q_j$ for some $j$, then $p_{i+1}$ is any point on the line opposite of $q_j$.
\end{itemize}
If we write $P = \oplus_{n\in \N} \C e_n$, then the point $p_i= [a_i:b_i:c_i]$ corresponds defining an action of $S$ on $P$ by
\begin{eqnarray*}
x \cdot e_i = a_i e_{i+1}, & y \cdot e_i = b_i e_{i+1}, &z \cdot e_i = c_i e_{i+1}.
\end{eqnarray*}
\begin{theorem}
The point modules of $T_t$ for $t \neq 0,\infty$ are parametrized by 6 lines, call this set $\mathbb{W}$. The isomorphism $\phi$ induced by sending a point module $P$ to $P[1]$ on $\mathbb{W}$ is such that $\phi^2$ fixes the intersection points and sends each line of $\mathbb{W}$ to itself.
\end{theorem}
\begin{proof}
Let $P$ be a point module of $T_t$ with associated point sequence $p_0 p_1 p_2 \ldots$. Let $p_i p_{i+1} p_{i+2}$ be a subtriple of this sequence.
\begin{itemize}
\item Assume that $p_i$ is one of the intersection points. Using the Heisenberg action, we may assume that $p_i = [1:0:0]$. Then $p_{i+1} = [0:\alpha:\beta]$.
\begin{itemize}
\item Assume that $[0:\alpha:\beta] \neq [0:1:0]$ and $[0:\alpha:\beta] \neq [0:0:1]$. Then $p_{i+2} = p_i$ and the relations of $S/(I_t)$ are trivially fulfilled.
\item Assume that $[0:\alpha:\beta] = [0:1:0]$. Then $p_{i+2} = [a:0:b]$. But from the degree 3 relations it follows that $b = 0$ and so $p_{i+2} = p_i$.
\item The case $[0:\alpha:\beta] = [0:0:1]$ is similar to the previous case.
\end{itemize}
\item Assume that $p_i$ is not one of the intersection points. Again using the Heisenberg-action, we may assume that $p_i = [0:\alpha:\beta]$, $p_{i+1}=[1:0:0]$ and $p_{i+2} = [0:\gamma:\delta]$. But then it follows from the degree 3 relations that
$ \delta \alpha = -t \gamma \beta$ or differently put, $\phi^2$ is an isomorphism of $\mathbf{V}(XYZ)$ fixing the intersection points.
\end{itemize}
\end{proof}
From the proof of this theorem we also notice that if $-t$ is primitive $n$th root of unity, then $\phi^{2n}$ is the identity. This implies that a point module $P$ of $T_t$ in this case parametrizes a $\C^*$-family of $2n$-dimensional simple representations of $T_t$.
\par The point module corresponding to $q_0 q_1 q_0 q_1 \ldots$ corresponds to the $\C^*$-family of 2-dimensional simple representations coming from the quotient $T_t/(z) = \C \langle x,y \rangle /(x^2,y^2)$. Similar results hold for the Heisenberg-orbit of this point module.
\subsection{The central element}
One of the many similarities we want for these new algebras is that there exists a central element of degree 3, fixed by the action of the Heisenberg group.
\begin{theorem}
For every element $t \in \Delta \subset \PP^1 \times \PP^1$ there exists a degree 3 central element in $(T_t)_3$ fixed by $H_3$.
\end{theorem}
\begin{proof}
Using for example MAGMA, one checks that
$$
g_t = (zxy + xyz + yzx)+t(yxz+zyx+xzy)
$$
is central in $T_t$.
\end{proof}
One also checks that $g_t$ acts trivially on each point module. These observations show that the constructed quotients have indeed much in common with the AS-regular algebras, as each point module of the generic AS-regular algebra is also annihilated by the unique central element in degree 3.
\subsection{An analogue of the twisted coordinate ring}
We can now calculate the Hilbert series of $M_t = T_t/(g_t)$.
\begin{theorem}
The Hilbert series of $M_t$ is $\frac{1-t^3}{(1-t)^3}$ except if $t=0,\infty$.
\end{theorem}
\begin{proof}
Adding the relation $(zxy + xyz + yzx)+t(yxz+zyx+xzy)$, we can find an easier basis for the degree 3 relations of $M_t$. Taking linear combinations, we find
$$
\begin{cases}
zxy = -t yxz,\\
xyz = -t zyx, \\
yzx = -t xzy,
\end{cases}
$$
for $t \neq 0$. For $t=-1$, the corresponding algebra is a quotient of the Clifford algebra, so in this particular case the Hilbert series is correct. But it then follows that the Hilbert series is correct for all these cases, as all these algebras are monomial algebras with the same basis as the quotient of the Clifford algebra.
\end{proof}
\begin{corollary}
$g_t$ is not a zerodivisor of $T_t$.
\end{corollary}
\begin{proof}
This follows directly from Hilbert series considerations and from Theorem \ref{th:Hilbert}.
\end{proof}
\par We can find an easy basis for $M_t=T_t/(g_t)$ if $t \neq 0,\infty$.
\begin{lemma}
A monomial $w \in M_t$ is 0 iff any letter is both on an even and an odd place in $w$.
\end{lemma}
\begin{proof}
If $w$ is a monomial with not one variable on both an even and an odd place, then $w$ can never be 0, as the 3 relations preserve the even and odd places of the variables. So the combinations $x^2,y^2,z^2$ can never occur in $w$.
\par Suppose now that $x$ occurs on both an even and an odd place. There is then a submonomial of $w$ of the form $xzyzy\ldots yx$ or $xyzyz\ldots zx$. But then we can bring the last $x$ to place 2 in the submonomial, so we get $x^2$ and $w$ becomes 0. Similar results are true for $y$ and $z$.
\end{proof}
If we now take a monomial which is not 0 in $M_t$, then it follows that either all odd or all even places have the same variable, say $x$. But then we can `jump' over this variable to get 1 variable to the left and the other to the right. Fixing a lexicographical order $x>y>z$, we can therefore find a nice basis in the following way: take 2 variables, say for example $x,y$, take the basis of $\C[x,y]$ an put a $z$ between each variable. Put $z$ first on the odd places and then the even places. Apply this procedure 3 times (one for each variable), but remember to discard 6 elements (for $xzxz\ldots$, $yzyzyzy\ldots$ and others have been counted twice). For example, a basis for $(M_t)_4$ is given by
$$
\begin{array}{ccc}
xyxy & xyzy & zyzy, \\
xzxz & xzyz & yzyz, \\
yxyx & yxzx & zxzx, \\
yxyz & zxzy & xyxz.
\end{array}
$$
\par In the even cases, we then find a linearly independent set consisting of $3(\frac{n}{2}+1) + 3(\frac{n}{2}+1)-6 = 3n$ elements, so this is a basis in even degree. In odd degree, we find $3(\frac{n+1}{2}+1) + 3 (\frac{n-1}{2}+1) - 6 = 3n$, so also in this case the constructed linearly independent set is a basis.
\par The next question is to determine the center of $M_t$.
\begin{proposition}
$M_t$ has no central elements in odd degree.
\end{proposition}
\begin{proof}
$M_t$ maps surjectively to the algebras $\C\langle y,z \rangle/(y^2,z^2)$, $\C\langle x,z \rangle/(x^2,z^2)$ and $\C\langle x,y \rangle/(x^2,y^2)$. This means that any central element of odd degree, say $w$, has to have all variables in its monomials, as each of these quotients doesn't have central elements in odd degree, so $w$ has to belong to the kernel of each of these algebra morphisms. So say that $w$ has as one of its monomials $xyxy\ldots x z x$, with all $x$ on the odd places. For $w$ to belong to the center, we need to get the first $y$ in $y xyxy\ldots x z x$ to the last place. This is however impossible, as the first $y$ is in a odd place and the last place is even. Similar results hold for monomials of the form $yxyx \ldots zy$ and $zxzx \ldots yz$. For monomials of the form $yx\ldots xz$ with $x$ on every even place, we have to get the first $x$ in $xyxyx\ldots xz$ to the last place, but again this is impossible as the first place is odd and the last place is even. Similar results hold for $xy \ldots yz$ ($y$ at even place) and $yz\ldots zx$ ($z$ at even place). This means that $w$ does not contain any monomials, so $w=0$.
\end{proof}
\begin{theorem}
If $-t$ is a primitive $n$th root of unity, then the elements $(x+y)^{2n}, (x+z)^{2n}$ and $(y+z)^{2n}$ generate the center.
\end{theorem}
\begin{proof}
It is clear that these 3 elements belong to the center of $M_t$, as we have that $$(x+y)^2x = x(x+y)^2,(x+y)^2y = y(x+y)^2,(x+y)^{2n}z = z(x+y)^{2n}.$$ The other claimed elements are central by the fact that the center is stable under the action of the Heisenberg group.
\par Consider the second Veronese subalgebra $P_t=M_t^{(2)}$ with generators $xy,yx,xz,zx,yz,zy$. We then have 
$$\begin{cases}
(xy)(yx) = 0 ,\\
(xy)(xz) = (-t)^{-1} (xz)(xy),\\
(xy)(zx) = 0, \\
(xy)(yz) = 0,\\
(xy)(zy) = (-t) (zy)(xy),
\end{cases}
$$
together with Heisenberg orbits (in particular, any word consisting of 3 different couples is necessarily 0). It then follows that the center of $P_t$ is generated by these 6 monomials of degree 2 to the $n$th power. Take a degree $2n$ central element $w$ of $M_t$ and suppose it contains the monomial $(xy)^{na}(xz)^{nb}$. We then have
$$
y(xy)^{na}(xz)^{nb} = (yx)^{na} y (xz)^{nb} = (yx)^{na} (zx)^{nb} y.
$$
We find that this monomial can only be in $w$ if and only if the monomial $(yx)^{na} (zx)^{nb}$ also occurs in $w$, with the same coefficient. But this sum of 2 monomials is equal to
$$
(xy+yx)^{na}(xz+zx)^{nb}.
$$
Similar results hold for the other monomials, so we are done.
\end{proof}
From now on, assume that $-t$ is a primitive $n$th root of unity. We have $(xy+yx)^{n}(yz+zy)^{n}(zx+xz)^{n}=0$, so the dimension of the center is $\leq 2$.
\begin{theorem}
$M_t$ is a finite module over its center.
\end{theorem}
\begin{proof}
$M_t$ is a finite module over $P_t$, which in turn is a finite module over its center. The center of $P_t$ is generated by $u=(xy)^n,u'=(yx)^n,v'=(xz)^n,v=(zx)^n,w=(yz)^n,w'=(zy)^n$. The center of $M_t$ is then equal to $u+u',v+v',w+w'$. Now, $Z(P_t)$ is generated as a module over $Z(M_t)$ by the elements $1,u,u',v,v',w,w'$, as we have for example
$$
u^av'^b = (u+u')(v+v')(u^{a-1}v'^{b-1})
$$
so by induction we conclude that $M_t$ is a finite module over its center.
\end{proof}
\begin{corollary}
The dimension of the center is 2.
\end{corollary}
\begin{proof}
We know that the dimension of the center is $\leq 2$. If it was $< 2$, then $M_t$ would not be a finite module over its center.
\end{proof}

We can now give a description of the simple representations of $M_t$.
\begin{theorem}
The PI-degree of $M_t$ is $2n$.
\end{theorem}
\begin{proof}
As the elements of the center of smallest degree are of degree $2n$, it follows from the Cayley-Hamilton polynomial that the PI-degree of $M_t$ is at least $2n$. If we can now find an open subset of $\wis{Spec}(Z(M_t))$ with $2n$-dimensional simple representations, we are done. However, by determining the point modules of $T_t$ and the observation that each point module is annihilated by $g_t$, we have indeed found a 2-dimensional family of $2n$-dimensional simple representations of $M_t$, as the induced automorphism $\phi$ on the point variety of $M_t$ coming from the shift functor of $\wis{Proj}(M_t)$ has the property $\phi^{2n}= Id$.

\end{proof}

\subsection{Connection with the Clifford algebra}
Let $\lambda \in \C^*$, for the moment not a root of unity. Define an action of the infinite cyclic group $\Z = C_\infty = \langle \sigma \rangle $ on the Clifford algebra $C$ by $$
\sigma(x) = \lambda^{-1}x, \sigma(y) = y,\sigma(z) = \lambda z
$$
and take the smash product $Q=C \# \Z$, which is graded by the classic grading on $C$ and $\deg(\sigma) = 0$. Consider the elements $x'=x \# 1, y'=y \# \sigma, z'= z\# \sigma^{-1}$. We then have
\begin{gather*}
z'x'y' = (z\# \sigma^{-1})(x\# 1)(y\# \sigma) = \lambda (zxy \# 1), \\
x'y'z' = (x\# 1)(y\# \sigma)(z\# \sigma^{-1}) = \lambda (xyz \# 1), \\
y'z'x' = (y\# \sigma)(z\# \sigma^{-1})(x\# 1) = \lambda (yzx \# 1), \\
y'x'z' = (y\# \sigma)(x\# 1)(z\# \sigma^{-1}) = (yzx \# 1),\\
z'y'x' = (z\# \sigma^{-1})(y\# \sigma)(x\# 1) = (zyx \# 1), \\
x'z'y' = (x\# 1)(z\# \sigma^{-1})(y\# \sigma) = (xzy \# 1).
\end{gather*}
This means that $x',y',z'$ are solutions for the relations of $T_{-\lambda}$.
\begin{proposition}
The Hilbert series of the algebra $Q$ generated by $x',y',z'$ is equal to $\frac{1}{(1-t)^3}$ and all the relations holding in this algebra come from degree 3 relations that hold in $T_{-\lambda}$.
\end{proposition}
\begin{proof}
The algebra $\mathcal{C}=C \# \Z$ is defined by generators and relations by
$$
\mathcal{C}=\C\langle x,y,z ,t ,t^{-1} \rangle/(x^2,y^2,z^2,[\{x,y\},z],[\{y,z\},x],tx - \lambda^{-1}xt,ty - yt,tz - \lambda zt).
$$
We give $\mathcal{C}$ the gradation determined by $\deg(x)=\deg(y)=\deg(z)=1,\deg(t)=0$. Then a basis in degree $k$ is determined by fixing a monomial basis $W_k$ for $C$ in degree $k$ and taking all powers of $t$
$$
\{f \# t^m | f \in W_k, m \in \Z\}.
$$
Now, write an element $f \in W_k$ as $f(x,y,z)$. Then the elements $f(x',y',z')$ are also linearly independent: let $a$ reps. $b,c$ be the number of times $x$, resp. $y,z$ is in $f(x,y,z)$. Then $f(x',y',z')$ is equal to $\lambda^{\mu} f(x,y,z) \# t^{b-c}$ for some $\mu \in \Z$.
\par The set $\{f(x',y',z') | f \in W_k \}$ forms a basis of $Q_k$: let $g(x,y,z)\#t^{r}$ be an element of $Q_k$ with $g(x,y,z)$ a monomial of degree $k$, $g\neq 0$. Then necessarily $r=b-c$ where $b$ is the number of times $y$ occurs in $g$ and $c$ is the number of times $z$ occurs in $g$. $g$ can be written as a unique linear combination of elements in $W_k$, say using the terms $f_1,\ldots,f_i$. However, due to the fact that in the defining relations of $C$ the number of times a variable occurs does not change (unless the monomial becomes $0$), this linear combination has the following property: for each $f_i$, the number of occurrences of $x,y,z$ stays the same. But then $g(x,y,z)\#t^{b-c}$ can be written as a linear combination of $f_i(x,y,z) \# t^{b-c} = \lambda^{-\mu}f(x',y',z')$. So the Hilbert series indeed stays the same.
\par The fact that the only relations are of degree $3$ follows as the only relations between monomials can occur if all variables occur and the relations come from $[\{x,y\},z]$ and $[\{y,z\},x]$.
\end{proof}
\begin{corollary}
The algebra $T_{-\lambda}$ can be embedded in a smash product $C \# \Z$.
\end{corollary}
From this, it follows that Theorem \ref{th:Hilbert} is proved.
\par When $\lambda$ is a primitive root of unity of order $n$, we can take the same action of $\Z_n = \langle \sigma \rangle$ on $C$, take the smash product $C \# \Z_n$ and take the same elements $x',y',z'$ as generators of a subalgebra of $C \# \Z_n$. These 3 elements then fulfil the relations of $T_{-\lambda}$ and again we have found a subalgebra of $P$ isomorphic to $T_{-\lambda}$.
\par In the case that $\lambda$ is a primitive $n$th root of unity, we can lift the $3$ linearly independent central elements of degree $n$ of $T_{-\lambda}/(g_{-\lambda})$.
\begin{proposition}
The element $(x'z'+z'x')^n$ belongs to the center of $C \# \Z_n$.
\end{proposition}
\begin{proof}
We have
$$
x'z'+z'x' = (x \# 1)(z \# \sigma^{-1})+(z \# \sigma^{-1})(x \# 1) = (xz+\lambda zx)\# \sigma^{-1}.
$$
$xz$ and $zx$ are fixed by $\sigma$, so we get 
$$
((xz+\lambda zx)\# \sigma^{-1})^n=(xz)^n+\lambda^n(zx)^n \# 1= (xz+zx)^n \# 1
$$
which is indeed central.
\end{proof}
This implies that $(xz+zx)^n$ belongs to the center of $T_{-\lambda}$. Then we can use the Heisenberg action in $T_{-\lambda}$ to see that $(yz+zy)^n$ and $(xy+yx)^n$ are also central in $T_{-\lambda}$. From \cite[Lemma 3.6]{smith1994center}, we now deduce
\begin{theorem}
The center of $T_{-\lambda}$ with $\lambda$ a primitive $n$th root of unity is generated by 1 element of degree 3 $g_{-\lambda}$ and 3 linear independent elements of degree $2n$, say $u,v,w$, with one relation of the form $uvw = \alpha g^{2n}$ for some $\alpha \in \C^*$. $T_{-\lambda}$ is also a finite module over its center.
\end{theorem}
\begin{proof}
All the conditions of the mentioned lemma are satisfied: $g_{-\lambda}$ is regular and the image of $\C[u,v,w]$ generates the center of $M_{-\lambda}$. Moreover, as we have in $M_{-\lambda}$ that the relation in the center is $\overline{u}\overline{v}\overline{w} = 0$, it follows that the only relation is of the form $uvw = \alpha g^{2n}$. $\alpha$ is not 0 as $(x'z'+z'x')^n(x'y'+y'x')^n(y'z'+z'y')^n$ is not 0 in $C \# \Z_n$.
\end{proof}
\subsection{Representation theory}
Fix $\lambda$ a primitive root of unity of order $n$, $n \neq 1$. Let $P=P_\lambda$ be  the smash product of $C$ with $\Z_n$ defined in last section and let $Q = Q_\lambda = T_{-\lambda}$ be the subalgebra of $P$ coming from a quotient of $S$.
\begin{theorem}
The center of $P$ is isomorphic to the ring $\C[a,b,c,d,e]/(ae - d^2,bc-e^n)$.
\end{theorem}
\begin{proof}
The center of $C$ is generated by $(x+y)^2,(y+z)^2,(x+z)^2,g=zxy-yxz$ with one relation of the form $(x+y)^2(y+z)^2(x+z)^2 = \alpha g^2$ for some $\alpha \in \C^*$. The elements in $Z(C)$ fixed by $\sigma$ are $xz+zx$, $(yz+zy)^n$, $(yx+xy)^n$, $g$ and $(yz+zy)(yx+xy)$. But then the relations become
\begin{gather*}
(yz+zy)((xz+zx)(yx+xy)) = \alpha g^2,\\
(xz+zx)^n(yx+xy)^n = ((xz+zx)(yx+xy))^n.
\end{gather*}
Up to a scalar, these are the relations of the claimed ring.
\end{proof}
\begin{corollary}
The PI-degree of $P$ is $2n$.
\end{corollary}
\begin{proof}
$P$ is a free module of rank $n$ over $C$. $C$ is a module of rank $4$ over its center $Z(C)$, which in turn is a module of rank $n$ over its ring of invariants $Z(C)^{\Z_n}$. The claim follows.	
\end{proof}
\begin{proposition}
The PI-degree of $Q$ is $2n$.
\end{proposition}
\begin{proof}
We know that the PI-degree is at least $2n$, as we have found simple $2n$-dimensional representations coming from the quotient of $Q$ by the degree 3 central element $g_{-\lambda}$. As $Z(P)$ is a finite $Z(Q)$-module of rank $n$, every simple $2n$-dimensional representation of $P$ induces a representation of $Q$ and the induced map $\xymatrix{\wis{Spec}(Z(P)) \ar[r] & \wis{Spec}(Z(Q))}$ is surjective. But then there is an open subset of $\wis{Spec}(Z(Q))$ containing simple $2n$-dimensional representations. This proves the claim.
\end{proof}
In fact, by using the $(\C^*)^3$-action as gradation preserving algebra automorphisms, we find
\begin{proposition}
The Azumaya locus of $\wis{Spec}(Z(Q))$ is equal to $\wis{Spec}(Z(Q))$ minus 3 lines, the 2 by 2 intersections of the 3 planes $\mathbf{V}(u),\mathbf{V}(v),\mathbf{V}(w)$.
\end{proposition}
\begin{proof}
The $(\C^*)^3$-action divides $\wis{Spec}(Z(Q))$ into 8 orbits: 1 of dimension 3 (the largest), 3 of dimension 2, 3 of dimension 1 and 1 of dimension 0. The first orbit is open and therefore intersects $\wis{Azu}_{2n}(Q)$, but then this orbit is contained in $\wis{Azu}_{2n}(Q)$. The 3 orbits of dimension 2 are the ones coming from simple representations of $M_t$ of dimension $2n$, so these also belong to $\wis{Azu}_{2n}(Q)$. The orbits corresponding to lines can not belong to $\wis{Azu}_{2n}(Q)$, as there are simple 2-dimensional representations coming from the quotient $Q/(z)=\C\langle x,y \rangle/(x^2,y^2)$.
\end{proof}
\section{The algebra $S_{[1:0:0]}$}
In this case we can be brief: $S_{[1:0:0]}$ is a Zhang-twist of $S_{[0:0:1]}$ by the automorphism $x \mapsto z \mapsto y \mapsto x$. Recall that if $\mathcal{A}$ is a graded algebra with a graded automorphism $\phi$, then the Zhang twist is defined as the algebra with the same generators but with multiplication rule $a * b = a \phi^i(b)$ if $b \in \mathcal{A}_i$. A Zhang twist of a graded algebra preserves the Hilbert series.
\begin{proposition}
$S_{[1:0:0]}$ has a 1-dimensional family of quotients parametrized by $\C^*$ such that these quotients have the right Hilbert series. The relations of degree 3 are defined by
\begin{align*}
(v_1)_t = (zyx+\omega   xzy+ \omega^2 yxz) + t (y^3+ \omega   z^3+\omega^2 x^3),\\
(v_2)_t = (zyx+\omega^2 xzy+ \omega   yxz) + t (y^3+ \omega^2 z^3+\omega   x^3).
\end{align*}
The central element becomes $g_t=(zyx+xzy+ yxz) + t (y^3+ z^3+ x^3)$
\end{proposition}
\begin{proof}
We find that 
\begin{align*}
x*x*x = xzy,y*y*y = yxz,z*z*z=zyx,\\
z*y*x = zxy,y*x*z = yzx,x*z*y=xyz
\end{align*}
and the other 3 monomials become 0. It is then clear that the relations from the proposition are really the Zhang twists of the relations in $S_{[0:0:1]}$.
\par If we calculate $(zyx+xzy+ yxz) + t (y^3+ z^3+ x^3)$ in $S_{[1:0:0]}$, we find
$$
zxy+xyz+yzx+t(yxz+zyx+xzy),
$$
which is central in the original quotient of $S_{[0:0:1]}$. As this element is also fixed under $\phi$, we are done.
\end{proof}
\section{The controlling variety}
In order to get algebras with the correct Hilbert series up to degree 4, we replaced the point $[0:0:1]$ with a $\PP^1$. One of course hopes that the corresponding variety parametrizing these algebras is just the blow-up of $\PP^2$ in this point.
\begin{theorem}
The variety parametrizing $H_3$-deformations up to degree 3 of the polynomial ring $\C[x,y,z]$ with the correct Hilbert series up to degree 4 is the blow-up of $\PP^2$ in 12 points.
\end{theorem}
\begin{proof}
We will show that the variety $$Z\subset \Emb_{H_3}(V^*,(V^*)^3) \times \Emb_{H_3}(\chi_{1,0}^2,\chi_{1,0}^3) \cong² \PP^2 \times \PP^2$$ parametrizing nice quotients with the right multiplicity of  $\chi_{1,0}$ in degree 3 is the blow-up of $\PP^2$ in the $H_3$-orbit of $[0:0:1]$.
\par Let $R$ correspond to the relations
$$
\begin{cases}
a yz + b zy + c x^2,\\
a zx + b xz + c y^2,\\
a xy + b yx + c z^2.
\end{cases}
$$
Let $\phi$ be the subspace of $V\otimes V \otimes V$ generated by
$$
\begin{cases}
A(x^3+\omega y^3+\omega^2 z^3)+B(zxy + \omega xyz + \omega^2 yzx) + C (yxz + \omega zyx + \omega^2 xzy),\\
D(x^3+\omega y^3+\omega^2 z^3)+E(zxy + \omega xyz + \omega^2 yzx) + F (yxz + \omega zyx + \omega^2 xzy)
\end{cases}
$$
such that the matrix
$$
\begin{bmatrix}
A & B & C \\ D & E & F
\end{bmatrix}
$$
has rank 2.
Decomposing $V \otimes R + R \otimes V$ in $H_3$-representations, we find that the $(R,\phi)$ belongs to $Z$ if and only if the following matrix has rank 2
$$M=
\begin{bmatrix}
c & a \omega^2 & b \omega \\
c & a \omega & b \omega^2 \\
A & B & C \\
D & E & F
\end{bmatrix}.
$$
We may assume that $c = 1$, putting $a=1$ or $b=1$ gives similar results. Put $a_{01} = AE-BD, a_{20} = CD-AF, a_{12} = BF-CE$. $M$ has rank 2 if and only if
\begin{eqnarray*}
a_{12}+a\omega a_{20} + b \omega^2 a_{01} = 0,\\
a_{12}+a\omega^2 a_{20} + b \omega a_{01} = 0.
\end{eqnarray*}
From this it follows that $a a_{20} = b a_{01}$, which is indeed the equation for the blow-up of $\PP^2$ in $[a:b:c]=[0:0:1]$. The same result holds for the points $[1:0:0]$ and $[0:1:0]$, call $Z_{1,0}$ the blow-up of $\PP^2$ in these 3 points.
\par We could now do the same for the representation $\chi_{2,0}$ to get $Z_{2,0}$, which is also isomorphic to $\PP^2$ blown-up in 3 points. However, in order to get the right Hilbert series up to degree 4, we have seen that we need to take the `diagonal' $\Delta \subset Z_{1,0} \times_{\PP^2} Z_{2,0}$, which is of course just the blow-up of $\PP^2$  in 3 points.
\par For the other $9$ points, we can use the $\wis{SL}_2(3)$-action.
\end{proof}
\section{The bad case $t=0,\infty$}
In the blow-up of $\PP^2$ in 12 points, there are still 24 points where the Hilbert series of the corresponding algebras explodes. Still, it would be useful to know what the Hilbert series of these algebras are. As all these algebras are isomorphic to each other (or isomorphic to a Zhang twist) by courtesy of the $\wis{SL}_2(3)$-action, it is enough to calculate the Hilbert series of the algebra
$$
\mathcal{C} = \C \langle x,y,z \rangle /(x^2,y^2,z^2, zxy + \omega xyz + \omega^2 yzx,zxy + \omega^2 xyz + \omega yzx).
$$
\begin{proposition}
The element $g_0 = zxy + xyz + yzx$ fulfils the conditions
$$
g_0x = x g_0 = g_0y = y g_0=g_0z = z g_0=0.
$$
\end{proposition}
\begin{proof}
Using the Heisenberg action, it is enough to prove this for $x$. Then it is an easy computer calculation with for example MAGMA.
\end{proof}
It is therefore enough to compute the Hilbert series of
$$
\mathcal{A}=\mathcal{C}/(g_0) = \C \langle x,y,z \rangle /(x^2,y^2,z^2,xyz,yzx,zxy).
$$
This algebra has the advantage that 2 words $w$ and $w'$ can only be equal to each other if and only if $w = w' = 0$. From this it follows that each $\mathcal{A}_n$ has a unique monomial basis $B_n$ on which $e_1$ acts by a permutation of order 3, without fixed points. This means that
\begin{align*}
\# \{ w \in B_n | w \text{  begins with } x \} &= \# \{w \in B_n | w \text{ begins with } y \} \\ &= \# \{w \in B_n | w \text{ begins with } z \}.
\end{align*}
\par We know from Lemma \ref{lem:Hil} that the Hilbert series starts with the terms
$$
H_\mathcal{A}(t)= 1 + 3t + 6t^2 + 9t^3 + 15t^4 + \ldots
$$
\begin{theorem}
Let $f(t)=H_\mathcal{A}(t) = \sum_{n=0}^\infty a_n t^n$ be the Hilbert series of $H_\mathcal{A}(t)$ with $a_n = \dim \mathcal{A}_n$. Then the coefficients of $f(t)$ fulfil the recurrence relation $a_n = 2 a_{n-1} - a_{n-3}$ for $n \geq 4$.
\end{theorem}
\begin{proof}
For $n=4$, we are done by Lemma \ref{lem:Hil}. Let $n \geq 5$. Let
$$
\xymatrix{\mathcal{A}_{n-1} \ar[r]^-{f_x} & \mathcal{A}_n}
$$
be the linear map defined by $w \mapsto xw$. It then follows that $a_n = 3 \dim \im(f_x)$. So we need to calculate $\ker(f_x)$. Due to the relations, $\ker(f_x)$ is the subspace of $\mathcal{A}_{n-1}$ spanned by the words beginning with $x$ or $yz$. The space spanned by the words beginning with $x$ has dimension $\frac{a_{n-1}}{3}$. The space spanned by the words beginning with $yz$ is the image of the map 
$$
\xymatrix{\mathcal{A}_{n-3} \ar[r]^-{f_{yz}} & \mathcal{A}_{n-1}}.
$$
The kernel of $f_{yz}$ is the subspace spanned by the monomials in $\mathcal{A}_{n-3}$ starting with $x$ or $z$. So $\dim \im(f_{yz}) = a_{n-3} - \frac{2a_{n-3}}{3} = \frac{a_{n-3}}{3}$. This gives
\begin{align*}
a_n &= 3 \dim \im(f_x) \\
    &=3 (a_{n-1} - \dim \ker(f_x))\\
    &=3 (a_{n-1} - \left(\frac{a_{n-1}}{3} +\frac{a_{n-3}}{3}\right)\\
    &=2 a_{n-1}- a_{n-3}.
\end{align*}
\end{proof}
\begin{corollary}
For $n\geq 3$, we have the recurrence relation $a_n = a_{n-1} + a_{n-2}$
\end{corollary}
\begin{proof}
$n=3$ is trivial. By induction, we may assume that $a_{n-1} = a_{n-2} + a_{n-3}$. But then we get
$$
a_{n} = 2 a_{n-1}- a_{n-3} = a_{n-1} + (a_{n-2} + a_{n-3})- a_{n-3} = a_{n-1} + a_{n-2}.
$$
\end{proof}
\begin{corollary}
$$
H_\mathcal{A}(t) = \frac{3}{1-(t+t^2)} - 2
$$
\end{corollary}
Unfortunately, this is not exactly the Fibonacci sequence. However, if we take $\mathcal{B}= \mathcal{A}^{e_1}$, then we get
$$
b_n = \dim \mathcal{B}_n = F_n
$$
with $F_n$ the $n$th Fibonacci number starting from $F_0 = 1, F_1 = 1$.

\end{document}